\documentclass{amsproc}
\usepackage{amsmath}
\usepackage{amsfonts}

\theoremstyle{plain}

\newtheorem{definition}{Definition}

\newtheorem{lemma}{Lemma}

\newtheorem{theorem}{Theorem}
\numberwithin{equation}{section}
\input{tcilatex}

\begin{document}
\title{Integral Transformations Between Some Function Spaces On Time Scales}
\author{M.Seyyit SEYY\.{I}DOGLU}
\address{U\c{s}ak \"{U}niversitesi Rekt\"{o}rl\"{u}\u{g}\"{u} 1 Eyl\"{u}l
Kamp\"{u}s\"{u} 64200 - U\c{S}AK}
\email{seyyit.seyyidoglu@usak.edu.tr\\
nozkan.tan@usak.edu.tr}
\author{N.\"{O}zkan TAN}
\date{April 15, 2011}

\begin{abstract}
In this paper we defined some function spaces on time scale which are Banach
spaces respect to supremum norm. We study integral transformations which are
carry to some important properties between mentioned above function spaces.
\end{abstract}

\maketitle

\section{\protect\bigskip Introduction}

The calculus on time scales has been introduced by Aulbach and Hilger \cite%
{Aul,Hil} in order to unify discrete and continuous analysis. In \cite%
{Aul,Hil,Boh} the concept of integral on time scales is defined by means of
an antiderivative(or pre-antiderivative)of function is called Cauchy
integral. In \cite{Sai} the Darboux and in \cite{Gus1,Gus2,Gus3}the Riemann
definitions of the integral on time scales are introduced and main theorems
of the integral calculus are established. In \cite{Bo} the improper Riemann-$%
\Delta $ Integral is defined which are important in the study of dynamic
systems on infinite intervals and properties improper Riemann-$\Delta $
integral are established.

Firstly we can give some basic definitions and theorems about the theory of
time sclales and Riemann-$\Delta $ integration. A time scale $\mathbb{T}$ is
an arbitrary nonempty closed subset of the real numbers $\mathbb{R}$. The
time scale $\mathbb{T}$ is a complete metric space with the usual metric. We
assume throughout that a time scale $\mathbb{T}$ has the topology that it
inherits from the real numbers with the standart topology.

For $t\in \mathbb{T}$ we define the \textit{forward jump operator} $\sigma :%
\mathbb{T\rightarrow T}$ by 
\begin{equation*}
\sigma \left( t\right) :=\inf \left\{ s\in \mathbb{T}\text{:~}s>t\right\}
\end{equation*}%
while the \textit{backward jump operator} $\rho :\mathbb{T\rightarrow T}$ is
defined by%
\begin{equation*}
\rho \left( t\right) :=\sup \left\{ s\in \mathbb{T}\text{:~}s<t\right\} .
\end{equation*}

If $\sigma \left( t\right) >t$, we say that $t$ is \textit{right-scattered},
while if $\rho \left( t\right) <t~$we say that $t$ is \textit{left-scattered}%
. Points that are right-scattered and left-scattered at the same time are
called \textit{isolated.} Also, if $\sigma \left( t\right) =t$, then $t$ is
called \textit{right-dense, and if }$\rho \left( t\right) =t,$ then $t$ is
called \textit{left-dense}. Points that are right-dense and left-dense are
called \textit{dense. }The \textit{graininess function }$\mu :\mathbb{%
T\rightarrow }\left[ 0,\infty \right) $ is defined by 
\begin{equation*}
\mu \left( t\right) :=\sigma \left( t\right) -t.
\end{equation*}

For $a,b\in \mathbb{T}$ with $a\leq b$ we define the interval $\left[ a,b%
\right] $ in $\mathbb{T}$ by 
\begin{equation*}
\left[ a,b\right] =\left\{ t\in \mathbb{T}\text{: }a\leq t\leq b\right\} .
\end{equation*}

Open intervals and half-open intervals etc. are defined accordingly. (see 
\cite{Boh})

Let $a<b$ be points in $\mathbb{T}$ and $\left[ a,b\right] $ the closed
interval in $\mathbb{T}$. A \textit{partition} of $\left[ a,b\right] $ is
any finite ordered subset%
\begin{equation*}
P=\left\{ t_{0},t_{1},...,t_{n}\right\} \subset \left[ a,b\right] \text{ \ \
\ \ \ \ where \ \ \ \ \ }a=t_{0}<t_{1}<...<t_{n}=b.
\end{equation*}%
We denote the set of all partitions of $\left[ a,b\right] $ by $\mathcal{P}=%
\mathcal{P}(a,b).$

\begin{lemma}
\label{Pdel}\cite{Gus3}For every $\delta >0$ there exists a partition $%
P=\left\{ t_{0},t_{1},...,t_{n}\right\} \in \mathcal{P}(a,b)$ such that for
each $i\in \left\{ 1,2,...,n\right\} $ either $t_{i}-t_{i-1}\leq \delta $ or 
$t_{i}-t_{i-1}>\delta $ and $\rho (t_{i})=t_{i-1}.$
\end{lemma}

\begin{definition}
\cite{Gus3}We denote by $\mathcal{P}_{\delta }=\mathcal{P}_{\delta }(a,b)$
the set of all $P\in \mathcal{P}(a,b)$ that possess the property indicated
in Lemma \ref{Pdel}.
\end{definition}

\begin{definition}
\cite{Gus3}Let $f$ be a function on $\left[ a,b\right] $ and let $P=\left\{
t_{0},t_{1},...,t_{n}\right\} \in \mathcal{P}(a,b).$ In each interval $%
[t_{i-1},t_{i})$, where $1\leq i\leq n$, choose an arbitrary point $\xi _{i}$
and form the sum%
\begin{equation*}
S=\sum\limits_{i=1}^{n}f(\xi _{i})(t_{i}-t_{i-1}).
\end{equation*}

We call $S$ a Riemann sum of $f$ corresponding to $P\in \mathcal{P}.$ We say
that $f$ is Riemann integrable on $\left[ a,b\right] $ provided there exists
a number $I$ with the following property: For each $\varepsilon >0$ there
exists $\delta >0$ such that $\left\vert S-I\right\vert <\varepsilon $ for
every Riemann sum $S$ of $f$ corresponding to a partition $P\in \mathcal{P}%
_{\delta }$ independent of the way in which we choose $\xi _{i}\in \lbrack
t_{i-1},t_{i}),1\leq i\leq n.$ The number $I$ is called the Riemann $\Delta $%
-integral of $f$ on $\left[ a,b\right] $ and we write$\tint%
\nolimits_{a}^{b}f(t)\Delta t=I.$
\end{definition}

\begin{theorem}
\cite{Boh}Let $f:\mathbb{T}\rightarrow \mathbb{R}$ and let $t\in \mathbb{T}.$%
Then $f$ is Riemann $\Delta $-integrable from $t$ to $\sigma (t)$ and%
\begin{equation*}
\int\nolimits_{t}^{\sigma (t)}f(s)\Delta s=\mu (t)f(t).
\end{equation*}
\end{theorem}

\begin{theorem}
\cite{Boh}Let $a,b\in \mathbb{T}$. Then we have the following:

i) If $\mathbb{T=R}$, then a function $f$ on $[a,b]$ is Riemann $\Delta $%
-integrable from $a$ to $b$ if and only if $f$ is Riemann integrable on $%
[a,b]$ in the classical sense, and in this case%
\begin{equation*}
\int\nolimits_{a}^{b}f(t)\Delta t=\int\nolimits_{a}^{b}f(t)dt,
\end{equation*}%
where the integral on the right is the ordinary Riemann integral.

ii) If $\mathbb{T=Z}$, then every function $f$ defined on $\mathbb{Z}$ is
the Riemann $\Delta $-integrable from $a$ to $b$ and%
\begin{equation*}
\int\nolimits_{a}^{b}f(t)\Delta t=\left\{ 
\begin{array}{cc}
\sum\nolimits_{t=a}^{b-1}f(t) & a<b \\ 
0 & a=b \\ 
\sum\nolimits_{t=b}^{a-1}f(t) & a>b.%
\end{array}%
\right.
\end{equation*}
\end{theorem}

\begin{theorem}
\cite{Gus3}Let $f$ and $g$ integrable functions on $\left[ a,b\right] $ and
let $\alpha \in \mathbb{R}.$ Then

i) $\alpha f$ is integrable and $\tint\nolimits_{a}^{b}(\alpha f)(t)\Delta
t=\alpha \tint\nolimits_{a}^{b}f(t)\Delta t,$

ii) $f+g$ is integrable and $\tint\nolimits_{a}^{b}(f+g)(t)\Delta
t=\tint\nolimits_{a}^{b}f(t)\Delta t+\tint\nolimits_{a}^{b}g(t)\Delta t,$

iii) $fg$ is integrable.
\end{theorem}

\begin{theorem}
\cite{Gus3}Let $f$ be a function defined on $\left[ a,b\right] $ and let $%
c\in \mathbb{T}$ with $a<c<b.$ If $f$ is integrable from $a$ to $c$ and $c$
to $b$, then $f$ is integrable from $a$ to $b$ and $\tint%
\nolimits_{a}^{b}f(t)\Delta t=\tint\nolimits_{a}^{c}f(t)\Delta
t+\tint\nolimits_{c}^{b}f(t)\Delta t.$
\end{theorem}

\begin{theorem}
\cite{Gus3}If $\ f$ and $g$ are integrable on $\left[ a,b\right] $ and $%
f(t)\leq g(t)$ for all $[a,b),$ then $\tint\nolimits_{a}^{b}f(t)\Delta t\leq
\tint\nolimits_{a}^{b}g(t)\Delta t.$
\end{theorem}

\begin{theorem}
\cite{Gus3}If $f$ is integrable on $\left[ a,b\right] $ then so is $%
\left\vert f\right\vert $ and%
\begin{equation*}
\left\vert \int\nolimits_{a}^{b}f(t)\Delta t\right\vert \leq
\int\nolimits_{a}^{b}\left\vert f(t)\right\vert \Delta t.
\end{equation*}
\end{theorem}

Now, we assume that $\mathbb{T}$ is unbounded above and $a\in \mathbb{T}$.
Let us suppose that%
\begin{equation*}
\left\{ t_{k}:k=0,1,...\right\} \subset \mathbb{T}\text{ \ \ \ where \ \ \ \ 
}a=t_{0}<t_{1}<...\text{ \ and }\underset{k\rightarrow \infty }{\lim }%
t_{k}=\infty
\end{equation*}%
and the function $f$ $:$ $\left[ a,\infty \right) =\left\{ t\in \mathbb{T}%
:t\geq a\right\} $ $\rightarrow \mathbb{R}$ is Riemann $\Delta $-integrable
from $a$ to any point $A\in \mathbb{T}$ with $A\geq a.$ If the integral%
\begin{equation*}
F(A)=\dint\nolimits_{a}^{A}f(t)\Delta t
\end{equation*}%
approaches a finite limit as $A\rightarrow \infty $, we call that limit the 
\textit{improper integral of first kind} of $f$ from $a$ to $\infty $ and we
write%
\begin{equation}
\dint\nolimits_{a}^{\infty }f(t)\Delta t=\underset{A\rightarrow \infty }{%
\lim }\dint\nolimits_{a}^{A}f(t)\Delta t  \label{imp}
\end{equation}%
In such a case we say that the improper integral (\ref{imp}) exists or that
it is convergent. (see \cite{Boh,Bo}).

\section{Some Function Spaces and Integral Transformations}

Throughout the study we assume that all time scales are unbounded above. Let 
$\mathbb{T}$ be a such time scale and $[\beta ,\infty )\subset \mathbb{T}$.
We denote the set of all real valued functions defined on $[\beta ,\infty )$
which are Riemann $\Delta $-integrable on every bounded subintervals of $%
[\beta ,\infty )$ by $\mathcal{R}_{\mathbb{T}}[\beta ,\infty )$. Function
spaces $C_{\mathbb{T}}[\beta ,\infty )$ and $C_{\mathbb{T}}^{0}[\beta
,\infty )$ be as follows.%
\begin{eqnarray*}
C_{\mathbb{T}}[\beta ,\infty ) &=&\left\{ f\in \mathcal{R}_{\mathbb{T}%
}[\beta ,\infty ):\lim\limits_{t\rightarrow \infty }f(t)\text{ exist}\right\}
\\
C_{\mathbb{T}}^{0}[\beta ,\infty ) &=&\left\{ f\in \mathcal{R}_{\mathbb{T}%
}[\beta ,\infty ):\lim\limits_{t\rightarrow \infty }f(t)=0\right\} .
\end{eqnarray*}%
It is evident that $C_{\mathbb{T}}[\beta ,\infty )$ and $C_{\mathbb{T}%
}^{0}[\beta ,\infty )$ \ are Banach spaces with respect to the norm%
\begin{equation*}
\left\Vert f\right\Vert =\sup\limits_{a\leq t<\infty }\left\vert
f(t)\right\vert .
\end{equation*}%
Note that, if $\mathbb{T=[}\beta ,\infty )=\mathbb{N}$ then these function
spaces become the space $c$ of convergent sequences and $c_{0}$ of null
sequences respectively.

We consider time scales $\mathbb{T}_{1}$ and $\mathbb{T}_{2}$. Let $[\alpha
,\infty )\subset \mathbb{T}_{1}$ and $[\beta ,\infty )\subset \mathbb{T}%
_{2}. $We assume that $f\in \mathcal{R}_{\mathbb{T}_{2}}[\beta ,\infty )$
and function $K:[\alpha ,\infty )\times \lbrack \beta ,\infty )\rightarrow 
\mathbb{R}$ is Riemann $\Delta $-integrable with respect to the variable $t$
on every bounded subinterval $[\beta ,\infty )$ for each $x\in \lbrack
\alpha ,\infty )$ i.e., $K(x,\circ )\in \mathcal{R}_{\mathbb{T}_{2}}[\beta
,\infty )$ for each $x\in \lbrack \alpha ,\infty ).$ If integral 
\begin{equation}
(Lf)(x)=\dint\nolimits_{\beta }^{\infty }K(x,t)f(t)\Delta t  \label{int don}
\end{equation}%
exists for all $x\in \lbrack \alpha ,\infty )$ then we transform $f$ \ to $%
Lf:[\alpha ,\infty )\rightarrow \mathbb{R}$. We use notation $(X,Y)$ for all
bounded-linear operators from $X$ to $Y.$

\begin{theorem}
\label{c0}Let $[\alpha ,\infty )\subset \mathbb{T}_{1}$, $[\beta ,\infty
)\subset \mathbb{T}_{2}$ and $K:[\alpha ,\infty )\times \lbrack \beta
,\infty )\rightarrow \mathbb{R}$ be a function such that $K(x,\circ )\in 
\mathcal{R}_{\mathbb{T}_{2}}[\beta ,\infty )$ for each $x\in \lbrack \alpha
,\infty ).$ Suppose the following conditions are satisfied:

i) $\underset{x\rightarrow x_{0}}{\lim }\dint\nolimits_{\beta }^{\infty
}\left\vert K(x,t)-K(x_{0},t)\right\vert \Delta t=0$ \ , \ $\forall x_{0}\in
\lbrack \alpha ,\infty )\newline
$

ii) $M=\sup\limits_{\alpha \leq x<\infty }\dint\nolimits_{\beta }^{\infty
}\left\vert K(x,t)\right\vert \Delta t<\infty $\newline

iii) $\underset{x\rightarrow \infty }{\lim }$ $\dint\nolimits_{\beta
}^{y}\left\vert K(x,t)\right\vert \Delta t=0$ $\ ,$ $\ \forall y\in \lbrack
\beta ,\infty ).$\newline
Then $L\in (C_{\mathbb{T}_{2}}^{0}[\beta ,\infty ),C_{\mathbb{T}%
_{1}}^{0}[\alpha ,\infty )).~$Moreover $\left\Vert L\right\Vert =M$.
\end{theorem}

\begin{proof}
Let $f\in C_{\mathbb{T}_{2}}^{0}[\beta ,\infty )$ \ and we can assume that $%
\left\Vert f\right\Vert \neq 0$. By inequality%
\begin{equation*}
\left\vert (Lf)(x)-(Lf)(x_{0})\right\vert \leq \left\Vert f\right\Vert
\dint\nolimits_{\beta }^{\infty }\left\vert K(x,t)-K(x_{0},t)\right\vert
\Delta t
\end{equation*}%
and condition (i), function $Lf$ is countinous for all $x_{0}\in \lbrack
\alpha ,\infty )$. In order to show that $\lim\limits_{x\rightarrow \infty
}(Lf)(x)=0,$consider equality 
\begin{equation}
(Lf)(x)=\dint\nolimits_{\beta }^{y}K(x,t)f(t)\Delta
t+\dint\nolimits_{y}^{\infty }K(x,t)f(t)\Delta t.  \label{ana}
\end{equation}

Since $\lim\limits_{t\rightarrow \infty }f(t)=0$, we can choose real number $%
y$ such that 
\begin{equation*}
\left\vert f(t)\right\vert <\frac{\varepsilon }{2M}~\text{\ \ \ }
\end{equation*}%
for all \ $t\geq y$ and we have%
\begin{eqnarray}
\left\vert \dint\nolimits_{y}^{\infty }K(x,t)f(t)\Delta t\right\vert &\leq
&\dint\nolimits_{y}^{\infty }\left\vert K(x,t)f(t)\right\vert \Delta t 
\notag \\
&<&\frac{\varepsilon }{2M}~\dint\nolimits_{y}^{\infty }\left\vert
K(x,t)\right\vert \Delta t  \notag \\
&<&\frac{\varepsilon }{2}.  \label{ikinciterim}
\end{eqnarray}%
Because of (iii) there exists $x_{0}$ such that%
\begin{equation*}
\dint\nolimits_{\beta }^{y}\left\vert K(x,t)\right\vert \Delta t<\frac{%
\varepsilon }{2\left\Vert f\right\Vert }
\end{equation*}%
for all $x>x_{0}$. So we have%
\begin{eqnarray}
\left\vert \dint\nolimits_{\beta }^{y}K(x,t)f(t)\Delta t\right\vert &\leq
&\dint\nolimits_{\beta }^{y}\left\vert K(x,t)f(t)\right\vert \Delta t  \notag
\\
&<&\left\Vert f\right\Vert \dint\nolimits_{\beta }^{y}\left\vert
K(x,t)\right\vert \Delta t  \notag \\
&<&\frac{\varepsilon }{2}  \label{ilkterim}
\end{eqnarray}%
for all $x>x_{0}$. By (\ref{ana}),(\ref{ikinciterim}) and (\ref{ilkterim})we
obtain $\left\vert (Lf)(x)\right\vert <\varepsilon $ for all $x>x_{0}$.
Therefore $\ Lf\in C_{\mathbb{T}_{1}}^{0}[\alpha ,\infty ).$ Let us show
that $\left\Vert L\right\Vert =M$. Since $\left\Vert Lf\right\Vert \leq
M\left\Vert f\right\Vert $ for all $f\in C_{\mathbb{T}_{2}}^{0}[\beta
,\infty )$ we obtain $\left\Vert L\right\Vert \leq M$. Let arbitrary $%
\varepsilon >0$ be given. There exists $x_{0}\in \mathbb{[\alpha },\infty )$
such that,%
\begin{equation}
M-\frac{\varepsilon }{2}<\dint\nolimits_{\beta }^{\infty }\left\vert
K(x_{0},t)\right\vert \Delta t.  \label{supozl}
\end{equation}%
Because of $\int\nolimits_{\beta }^{\infty }\left\vert K(x_{0},t)\right\vert
\Delta t<\infty $ there exist $p\in (\beta ,\infty )$ such that,%
\begin{equation}
\dint\nolimits_{p}^{\infty }\left\vert K(x_{0},t)\right\vert \Delta t<\frac{%
\varepsilon }{2}.  \label{yak}
\end{equation}%
Consider function $f\in C_{\mathbb{T}_{2}}^{0}[\beta ,\infty )$ defined by%
\begin{equation*}
f(t)=\left\{ 
\begin{array}{cc}
0 & ,~~~~t>p \\ 
\text{sgn}K(x_{0},t) & ,~~~~t\leq p.%
\end{array}%
\right.
\end{equation*}%
It is clear that $\left\Vert f\right\Vert =1$. By (\ref{supozl}) and (\ref%
{yak}) we have%
\begin{equation*}
\left\Vert Lf\right\Vert =\sup\limits_{\alpha \leq x<\infty }\left\vert
(Lf)(x)\right\vert \geq \left\vert (Lf)(x_{0})\right\vert
=\dint\nolimits_{\beta }^{p}\left\vert K(x_{0},t)\right\vert \Delta
t>M-\varepsilon .
\end{equation*}%
So $\left\Vert L\right\Vert \geq M$. Therefore $\left\Vert L\right\Vert =M$.
\end{proof}

\begin{theorem}
\label{regular}Let $[\alpha ,\infty )\subset \mathbb{T}_{1}$, $[\beta
,\infty )\subset \mathbb{T}_{2}$ and $K:[\alpha ,\infty )\times \lbrack
\beta ,\infty )\rightarrow \mathbb{R}$ be a function such that $K(x,\circ
)\in \mathcal{R}_{\mathbb{T}_{2}}[\beta ,\infty )$ for each $x\in \lbrack
\alpha ,\infty ).$ Suppose the following conditions are satisfied:

i) $\underset{x\rightarrow x_{0}}{\lim }\dint\nolimits_{\beta }^{\infty
}\left\vert K(x,t)-K(x_{0},t)\right\vert \Delta t=0$ \ , \ $\forall x_{0}\in
\lbrack \alpha ,\infty )$

ii) $\sup\limits_{\alpha \leq x<\infty }\dint\nolimits_{\beta }^{\infty
}\left\vert K(x,t)\right\vert \Delta t<\infty $

iii) $\underset{x\rightarrow \infty }{\lim }\dint\nolimits_{\beta
}^{y}\left\vert K(x,t)\right\vert \Delta t=0$ \ $,$ $\ \forall y\in \lbrack
\beta ,\infty )$\newline

iv) $\underset{x\rightarrow \infty }{\lim }\dint\nolimits_{\beta }^{\infty
}K(x,t)\Delta t=1.$

Then $L\in (C_{\mathbb{T}_{2}}[\beta ,\infty ),C_{\mathbb{T}_{1}}[\alpha
,\infty )).$ Moreover if $f(t)\rightarrow s$ as $t\rightarrow \infty $ then $%
(Lf)(x)\rightarrow s$ as $x\rightarrow \infty .$
\end{theorem}

\begin{proof}
For $s=0$, it is evident by Theorem \ref{c0}. If $s\neq 0$ \ repeat by
Theorem \ref{c0}, we have%
\begin{equation*}
(Lf)(x)=\dint\nolimits_{\beta }^{\infty }K(x,t)\left( f(t)-s\right) \Delta
t+s\dint\nolimits_{\beta }^{\infty }K(x,t)\Delta t.
\end{equation*}%
So it is clear that if $f(t)\rightarrow s$ as $t\rightarrow \infty $ then $%
(Lf)(x)\rightarrow s$ as $x\rightarrow \infty .$
\end{proof}

\begin{theorem}
\label{teofonk}Let interval $[\beta ,\infty )$ occures isolated points of
time scale $\mathbb{T}$. If $F\in C_{\mathbb{T}}^{\ast }[\beta ,\infty )$
where is dual space of $C_{\mathbb{T}}[\beta ,\infty ).$ Then there exists
real number $b$ and $\ $sequence $(b_{n})\in l_{1}$ such that 
\begin{equation}
F(f)=b\underset{t\rightarrow \infty }{\lim }f(t)+\sum\limits_{n=1}^{\infty
}b_{n}f(t_{n})  \label{fonk}
\end{equation}%
for all $f\in C_{\mathbb{T}}[\beta ,\infty )$. Moreover, norm of the
functional $F$ is%
\begin{equation}
\left\Vert F\right\Vert =\left\vert b\right\vert +\sum\limits_{n=1}^{\infty
}\left\vert b_{n}\right\vert .  \label{fonknorm}
\end{equation}%
On the contrary, if real number $b$ and $\ $sequence $(b_{n})\in l_{1}$
given, left side of equality (\ref{fonk}) is a member of\ $C_{\mathbb{T}%
}^{\ast }[\beta ,\infty )$. Moreover $C_{\mathbb{T}}^{\ast }[\beta ,\infty )$
and $l_{1}$ are isomorphic spaces.
\end{theorem}

\begin{proof}
The contrary side is straightforward. Since members of the set $[\beta
,\infty )$ are isolated points, we can denote the set $[\beta ,\infty )$ by $%
[\beta ,\infty )=\left\{ t_{1},t_{2},...\right\} $ where $\beta
=t_{1}<t_{2}<...\ $and $t_{k}\rightarrow \infty $ as $k\rightarrow \infty .$
Let $F\in C_{\mathbb{T}}^{\ast }[\beta ,\infty )$. The set $\left\{
e,e_{1},e_{2},...\right\} $ is a Schauder basis for $C_{\mathbb{T}}[\beta
,\infty )$ where $e\equiv 1$ and $e_{i}(t_{j})=\delta _{ij}$ ($\delta _{ij}$
is Kronecker delta). For any member $f$ of $C_{\mathbb{T}}[\beta ,\infty )$
is expressed by%
\begin{equation*}
f=le+\sum\limits_{n=1}^{\infty }\left( f(t_{n})-l\right) e_{n}
\end{equation*}%
where $l=\underset{t\rightarrow \infty }{\lim }f(t)$. By linearity and
continuity of $F$ we have 
\begin{equation}
F(f)=lF(e)+\sum\limits_{n=1}^{\infty }\left( f(t_{n})-l\right) F(e_{n})
\label{yrdest}
\end{equation}%
for all $f\in C_{\mathbb{T}}[\beta ,\infty ).$ Consider function $f\in $ $C_{%
\mathbb{T}}[\beta ,\infty )$ defined by 
\begin{equation*}
f(t_{n})=\left\{ 
\begin{array}{cc}
\text{sgn}F(e_{n}) & 1\leq n\leq r \\ 
0 & n>r%
\end{array}%
\right. 
\end{equation*}%
for all $r\geq 1.$ Since $\left\Vert f\right\Vert =1$ and $\left\vert
F(f)\right\vert \leq \left\Vert F\right\Vert \left\Vert f\right\Vert $, we
have%
\begin{equation}
\left\vert F(f)\right\vert =\sum\limits_{n=1}^{r}\left\vert
F(e_{n})\right\vert \leq \left\Vert F\right\Vert   \label{we}
\end{equation}%
for all $r\geq 1.$ By (\ref{we}) we obtain%
\begin{equation*}
\sum\limits_{n=1}^{\infty }\left\vert F(e_{n})\right\vert \leq \left\Vert
F\right\Vert <\infty .
\end{equation*}%
It means that $\sum\nolimits_{n=1}^{\infty }F(e_{n})$ is absolute
convergent. Let $b=F(e)-\sum\nolimits_{n=1}^{\infty }F(e_{n})$ and $%
b_{n}=F(e_{n})$. By equality (\ref{yrdest}) we get 
\begin{equation}
F(f)=bl+\sum\limits_{n=1}^{\infty }b_{n}f(t_{n}).  \label{sd}
\end{equation}%
Since $\left\vert l\right\vert \leq \left\Vert f\right\Vert $ and by (\ref%
{sd}) we have%
\begin{equation*}
\left\vert F(f)\right\vert \leq \left( \left\vert b\right\vert
+\sum\limits_{n=1}^{\infty }\left\vert b_{n}\right\vert \right) \left\Vert
f\right\Vert .
\end{equation*}%
Therefore $\left\Vert F\right\Vert \leq \left\vert b\right\vert
+\sum\limits_{n=1}^{\infty }\left\vert b_{n}\right\vert $. Now consider to
function $\ f\in $ $C_{\mathbb{T}}[a,\infty )$ defined by%
\begin{equation*}
f(t_{n})=\left\{ 
\begin{array}{cc}
\text{sgn}b_{n} & 1\leq n\leq r \\ 
\text{sgn}b & n>r.%
\end{array}%
\right. 
\end{equation*}%
It is obvious that $\left\Vert f\right\Vert =1$ and $f(t)\rightarrow $sgn$b$
as $t\rightarrow \infty $. By $r\rightarrow \infty $ in%
\begin{equation*}
\left\vert F(f)\right\vert =\left\vert \left\vert b\right\vert
+\sum\limits_{n=1}^{r}\left\vert b_{n}\right\vert
+\sum\limits_{n=r+1}^{\infty }b_{n}\text{sgn}b\right\vert \leq \left\Vert
F\right\Vert ,
\end{equation*}%
we obtain%
\begin{equation*}
\left\vert b\right\vert +\sum\limits_{n=1}^{\infty }\left\vert
b_{n}\right\vert \leq \left\Vert F\right\Vert .
\end{equation*}%
For isomorphism $C_{\mathbb{T}}^{\ast }[\beta ,\infty )$ to $l_{1}$ consider
the operator $T:C_{\mathbb{T}}^{\ast }[\beta ,\infty )\rightarrow l_{1}$
defined by $T(F)=(b,b_{1},b_{2},...)$, it is evident that $\left\Vert
T(F)\right\Vert =\left\vert b\right\vert +\left\vert b_{1}\right\vert
+\left\vert b_{2}\right\vert +...=\left\Vert F\right\Vert $ so $T$ preserves
norm.
\end{proof}

\begin{theorem}
\label{CT0}Let interval $[\beta ,\infty )$ occures isolated points of time
scale $\mathbb{T}_{2}$ and $[\alpha ,\infty )$ be a subinterval of $\mathbb{T%
}_{1}.$ If $L\in (C_{\mathbb{T}_{2}}^{0}[\beta ,\infty ),C_{\mathbb{T}%
_{1}}^{0}[\alpha ,\infty ))$ then there exists the function $K:[\alpha
,\infty )\times \lbrack \beta ,\infty )\rightarrow \mathbb{R}$ such that $%
K(x,\circ )\in \mathcal{R}_{\mathbb{T}_{2}}[\beta ,\infty )$ for each $x\in
\lbrack \alpha ,\infty )$ which is satisfied equality (\ref{int don}).
Moreover, has the following properties:

i) $\left\Vert L\right\Vert =\sup\limits_{\alpha \leq x<\infty
}\dint\nolimits_{\beta }^{\infty }\left\vert K(x,t)\right\vert \Delta
t<\infty $\newline

ii) $\underset{x\rightarrow \infty }{\lim }\dint\nolimits_{\beta
}^{y}\left\vert K(x,t)\right\vert \Delta t=0$ \ $,$ $\ \forall y\in \lbrack
\beta ,\infty ).$\newline
\end{theorem}

\begin{proof}
Let $[\beta ,\infty )=\left\{ t_{1},t_{2},...\right\} $ where $\beta
=t_{1}<t_{2}<...\ $and $t_{k}\rightarrow \infty $ as $k\rightarrow \infty $.
The set $\left\{ e,e_{1},e_{2},...\right\} $ is a Schauder basis for $C_{%
\mathbb{T}_{2}}^{0}[\beta ,\infty )$. For any member $f$ of $C_{\mathbb{T}%
_{2}}^{0}[\beta ,\infty )$ is expressed by 
\begin{equation*}
f=\dsum\limits_{k=1}^{\infty }f(t_{k})e_{k}.
\end{equation*}%
Let $b_{k}(x)=(Le_{k})(x)$ and the function $K:[\alpha ,\infty )\times
\lbrack \beta ,\infty )\rightarrow \mathbb{R}$ defined by $%
K(x,t_{k})(t_{k+1}-t_{k})=b_{k}(x)$. Since $L\in (C_{\mathbb{T}%
_{2}}^{0}[\beta ,\infty ),C_{\mathbb{T}_{1}}^{0}[\alpha ,\infty ))$ we have,

\begin{eqnarray}
(Lf)(x) &=&L\left( \dsum\limits_{k=1}^{\infty }f(t_{k})e_{k}\right) (x) 
\notag \\
&=&\dsum\limits_{k=1}^{\infty }f(t_{k})(Le_{k})(x)  \notag \\
&=&\dsum\limits_{k=1}^{\infty }f(t_{k})b_{k}(x)  \label{nrm} \\
&=&\dsum\limits_{k=1}^{\infty }K(x,t_{k})f(t_{k})(t_{k+1}-t_{k})  \notag \\
&=&\dint\nolimits_{\beta }^{\infty }K(x,t)f(t)\Delta t.  \label{intgst}
\end{eqnarray}%
By the hypothesis, we know that $Lf\in C_{\mathbb{T}_{1}}^{0}[\alpha ,\infty
)$ for all $f\in C_{\mathbb{T}_{2}}^{0}[\beta ,\infty )$. Hence, it is clear
that $Le_{k}\in C_{\mathbb{T}_{1}}^{0}[\alpha ,\infty )$. This means that
the function $K(x,t_{k})\rightarrow 0\ $as $x\rightarrow \infty $ for all $%
t_{k}$. Thus we obtain (ii). Let us show that $\left\Vert L\right\Vert
=\sup\limits_{\alpha \leq x<\infty }\int\nolimits_{\beta }^{\infty
}\left\vert K(x,t)\right\vert \Delta t$. By Theorem \ref{c0}, it is only
need to show this supremum exists. Functionals $L_{x}:C_{\mathbb{T}%
_{2}}^{0}[\beta ,\infty )\rightarrow \mathbb{R}$ defined by $L_{x}(f)=(Lf)(x)
$ are linear for each $x\in \lbrack \alpha ,\infty )$. Since%
\begin{equation*}
\left\vert L_{x}(f)\right\vert =\left\vert (Lf)(x)\right\vert \leq
\left\Vert Lf\right\Vert \leq \left\Vert L\right\Vert \left\Vert
f\right\Vert 
\end{equation*}%
$L_{x}$ are bounded for each $x\in \lbrack \alpha ,\infty )$. By the uniform
boundedness principle, we have%
\begin{equation*}
\underset{\alpha \leq x<\infty }{\sup }\left\Vert L_{x}\right\Vert <\infty .
\end{equation*}%
Hence, by the Theorem \ref{teofonk} norm of functionals $L_{x}$ which have
form of (\ref{nrm}) is 
\begin{equation*}
\left\Vert L_{x}\right\Vert =\dsum\limits_{k=1}^{\infty }\left\vert
b_{k}(x)\right\vert .
\end{equation*}%
Therefore we obtain,%
\begin{eqnarray*}
\dsum\limits_{k=1}^{\infty }\left\vert b_{k}(x)\right\vert 
&=&\dsum\limits_{k=1}^{\infty }\left\vert K(x,t_{k})\right\vert
(t_{k+1}-t_{k}) \\
&=&\dint\nolimits_{\beta }^{\infty }\left\vert K(x,t)\right\vert \Delta t.
\end{eqnarray*}
\end{proof}

\begin{theorem}
Let interval $[\beta ,\infty )$ occures isolated points of time scale $%
\mathbb{T}_{2}$ and $[\alpha ,\infty )$ be a subinterval of $\mathbb{T}_{1}.$
If $L\in (C_{\mathbb{T}_{2}}[\beta ,\infty ),C_{\mathbb{T}_{1}}[\alpha
,\infty ))$ and $(Lf)(x)\rightarrow s$ as $x\rightarrow \infty $ whenever $%
f(t)\rightarrow s$ as $t\rightarrow \infty $ for all $f\in C_{\mathbb{T}%
_{2}}[\beta ,\infty )$ then there exists the function $K:[\alpha ,\infty
)\times \lbrack \beta ,\infty )\rightarrow \mathbb{R}$ such that $K(x,\circ
)\in \mathcal{R}_{\mathbb{T}_{2}}[\beta ,\infty )$ for each $x\in \lbrack
\alpha ,\infty )$ which is satisfied equality (\ref{int don}). Moreover, $K$
has the following properties:

i) $\left\Vert L\right\Vert =\sup\limits_{\alpha \leq x<\infty
}\dint\nolimits_{\beta }^{\infty }\left\vert K(x,t)\right\vert \Delta
t<\infty $\newline

ii) $\underset{x\rightarrow \infty }{\lim }\dint\nolimits_{\beta
}^{y}\left\vert K(x,t)\right\vert \Delta t=0$ \ $,$ $\ \forall y\in \lbrack
\beta ,\infty )$\newline

iii) $\underset{x\rightarrow \infty }{\lim }\dint\nolimits_{\beta }^{\infty
}K(x,t)\Delta t=1.$
\end{theorem}

\begin{proof}
(i) and (ii) are obvious by Theorem \ref{CT0}. For (iii) we can take
constant function $f\equiv 1$ in (\ref{intgst}).
\end{proof}

\end{document}